\RequirePackage{plautopatch}
\documentclass[reqno]{amsart}
\usepackage{mathtools}
\usepackage{amsmath}
\usepackage{amsthm}
\usepackage{cleveref}
\usepackage{enumitem}
\setlist[description]{leftmargin=0pt, labelindent=0pt}

\title{A Complete Proof of the Limit Formula for Observable Diameter}
\date{\today}
\author{Shigeaki Yokota}
\address{Graduate School of Science, Tohoku University, Sendai 980-8578, Japan}
\email{shigeaki.yokota.t4@dc.tohoku.ac.jp}
\thanks{This work was partially supported by JSPS KAKENHI Grant Number 22H04942, for which the author served as a research assistant.}

\theoremstyle{plain}
\newtheorem{theorem}{Theorem}[section]

\newtheorem{lemma}[theorem]{Lemma}
\newtheorem*{theorem*}{Theorem}
\theoremstyle{definition}
\newtheorem{definition}[theorem]{Definition}

\theoremstyle{remark}
\newtheorem{remark}[theorem]{Remark}
\newtheorem{example}[theorem]{Example}
\newtheorem*{remark*}{Remark}

\counterwithin{equation}{section}

\newcommand{\R}{\mathbb{R}}
\newcommand{\X}{\mathcal{X}}
\renewcommand{\P}{\mathcal{P}}
\newcommand{\Q}{\mathcal{Q}}
\newcommand{\M}{\mathcal{M}}
\newcommand{\dpr}{d_{\operatorname{P}}}
\DeclareMathOperator{\lip1}{Lip_1}
\DeclareMathOperator{\id}{id}
\DeclareMathOperator{\pd}{PartDiam}
\DeclareMathOperator{\od}{ObsDiam}
\DeclareMathOperator{\diam}{diam}

\newcommand{\hausp}[1]{{\left(#1\right)_{\operatorname{H}}}}
\newcommand{\mmid}{\mathrel{}\middle\vert\mathrel{}}

\begin{document}

\begin{abstract}
    Ozawa and Shioya~\cite{ozawa2015limit} proposed the limit formula for observable diameters of pyramids under weak convergence. However, we find a constructive counterexample to an inequality used in their proof. In this paper, we correct the inequality and verify the limit formula.
\end{abstract}
\maketitle

\section{Introduction}
Gromov~\cite{gromov2007met} developed the geometry of mm-spaces, based on the concentration of measure phenomenon and the theory of collapsing manifolds. A triple $(X,d_X,\mu_X)$, or simply $X$, is called an \emph{mm-space} if $d_X$ is a complete separable metric on $X$ and $\mu_X$ is a Borel probability measure with full support in $(X,d_X)$. He introduced the observable distance between two mm-spaces, based on measure concentration. He also constructed a natural compactification of this distance, with each element termed a \emph{pyramid}.

A notable quantity in mm-space theory is the ($\kappa$-)\emph{observable diameter} of an mm-space $X$ with screen $Y$, defined as follows: Let $Y$ be a metric space. For any Borel probability measure $\mu$ on $Y$ and any real number $\alpha\leq 1$, the $\alpha$-\emph{partial diameter} of $\mu$, denoted by $\pd(\mu;\alpha)$, is defined as the infimum of the diameters of all Borel subsets $A\subset Y$ with $\mu(A) \geq \alpha$. Let $X$ be an mm-space and $\kappa \geq 0$. Let $\lip1(X,Y)$ denote the set of all 1-Lipschitz continuous functions from $X$ to $Y$. The ($\kappa$-)\emph{observable diameter} with screen $Y$ is defined as
\begin{equation*}
    \od(X,Y;-\kappa) \coloneqq \sup \{\ \pd(f_*\mu_X;1-\kappa) \mid f \in \lip1(X,Y) \ \},
\end{equation*}
where $f_*\mu_X$ is the push-forward of $\mu_X$ by $f$. Additionally, when $Y=\R$, we simply write it as $\od(X;-\kappa)$ and call it the ($\kappa$-)\emph{observable diameter} of $X$.

The observable diameter naturally extends to pyramid. Ozawa and Shioya~\cite{ozawa2015limit} claimed the following theorem.
\begin{theorem}\label{thm:pyramid_od_limit}
    Let $\P$ and $\P_n$, $n=1,2,\ldots$, be pyramids and let $\kappa > 0$. If $\P_n$ converges weakly to $\P$, then
    \begin{align*}
        \od(\P; -\kappa) & = \lim_{\varepsilon\to 0+}\liminf_{n\to\infty} \od(\P_n; -(\kappa+\varepsilon))   \\
                         & = \lim_{\varepsilon\to 0+} \limsup_{n\to\infty} \od(\P_n; -(\kappa+\varepsilon)).
    \end{align*}
\end{theorem}
\Cref{thm:pyramid_od_limit} has several important applications such as
the estimate of the observable diameter of $l^p$-product spaces,
the study of $N$-L\'evy families, and the phase transition property
(see \cite{ozawa2015product,ozawa2015limit}). For details on the weak convergence of pyramids, see Section~\ref{sec:preliminaries} or \cite[Definition~2.23]{ozawa2015limit}.

In the proof of the above theorem, specifically in the second equality, they rely on the following unproven inequality (see \cite[Lemma 3.10]{ozawa2015limit}):
\begin{equation}\label{eq:od-min}
    \min\{2R, \od(X; -\kappa)\} \leq \od\!\left(X,[-R,R]; -\kappa\right)
\end{equation}
for any $R > 0$ and $\kappa > 0$. However, we find the following counterexample: Let $R>0$ be a real number and define $X\coloneqq \{R,2R,3R,4R\}$ with the Euclidean metric on $\R$ and the normalized counting measure. In this case, we have
\begin{equation*}
    2R > \od(X; -\kappa) = R > \frac{2R}{3} = \od(X,[-R,R]; -\kappa)
\end{equation*}
for any $\kappa\in[1/2,3/4)$. Even if we modify the right-hand side to
\begin{equation*}
    \od(X, [-C(R), C(R)]; -\kappa)
\end{equation*}
using a constant $C(R)>0$ that depends only on $R$, there is still the following counterexample. For all $N=2,3,\ldots$, there exists an mm-space $X_N$ such that\begin{equation*}
    \od(X_N; -\kappa) = R > \od(X_N, [-(N-1)R, (N-1)R]; -\kappa)
\end{equation*}
for any $\kappa\in[1-1/N,1-1/2N)$. For details, see Example~\ref{example:odcounterexample}.

To complete the proof of Theorem~\ref{thm:pyramid_od_limit}, we prove the following theorem.
\begin{theorem}\label{thm:lip1minRpartdiam}
    Let $\alpha\in(0,1)$ and $R>0$ be two real numbers. For any Borel probability measure $\mu$ on $\R$, there exists a function $f \in \lip1(\R, [-R/\alpha, R/\alpha])$ such that
    \begin{equation*}
        \pd(f_*\mu; \alpha) = \min\{R, \pd(\mu; \alpha)\}.
    \end{equation*}
\end{theorem}
By \Cref{thm:lip1minRpartdiam}, the inequality (\ref{eq:od-min}) is revised as
\begin{equation*}
    \min\{R, \od(X; -\kappa)\} \leq \od\!\left(X,\left[-\frac{R}{1-\kappa},\frac{R}{1-\kappa}\right];\ -\kappa\right)
\end{equation*}
for any $\kappa\in(0,1)$ and $R>0$. This completes the proof of \Cref{thm:pyramid_od_limit}.
Furthermore, this revised inequality is sharp in the following sense.
\begin{theorem}\label{prop:lip1minRpartdiamIsSharp}
    Let $R>0$ be a real number. Then
    \begin{equation*}
        \limsup_{\kappa\to 1-}\ \sup \left\{ \frac{2R}{1-\kappa} - \diam I \mid I\subset \R \text{ Borel}, Q(R,I,\kappa) \leq 1\right\} \leq 2R,
    \end{equation*}
    where
    \begin{equation*}
        Q(R,I,\kappa) \coloneqq \sup_{X\in\X} \frac{\ \min\{R, \od(X; -\kappa)\}\ }{\od\left(X,I; -\kappa\right)}.
    \end{equation*}
\end{theorem}

\section{Preliminaries}\label{sec:preliminaries}
Let $\lip1(X,Y)$ denote the set of all 1-Lipschitz continuous maps from a metric space $X$ to another $Y$, where a map $f\colon X\to Y$ is 1-Lipschitz continuous if $d_Y(f(x),f(y)) \leq d_X(x,y)$ holds for all $x,y\in X$. Additionally, when $Y=\R$, we write it simply as $\lip1(X)$. For all subsets $A$ of a metric space $X$, we denote the diameter of $A$ by $\diam(A)$.

\begin{lemma}\label{lem:pdLowerSemiContinuous}
    Let $\mu$ be a Borel probability measure on a metric space $X$. Then, $\pd(\mu;\alpha)$ is lower semi-continuous in $\alpha \leq 1$.
\end{lemma}
\begin{proof}
    In $\alpha \leq 1$, we obtain that $\pd(\mu;\alpha)$ is monotonically non-decreasing from the definition and is left-continuous from \cite[Lemma 3.1(1)]{ozawa2015limit}. This completes the proof.
\end{proof}
\begin{lemma}[{\cite[Proposition 2.18 (1)]{shioya2016mmg}}]\label{lemma:lip1reductpartdiam}
    Let $\alpha \leq 1$ be a real number, $\mu$ a Borel probability measure on $\R$, and $f$ a function in $\lip1(\R)$. Then, $\pd(f_*\mu;\alpha) \leq \pd(\mu;\alpha)$.
\end{lemma}
\begin{lemma}\label{lemma:preservedistancepartdiam}
    Let $\alpha \leq 1$, $s$ and $c$ be real numbers, and let $\mu$ be a Borel probability measure on $\R$. Set a function $f$ as $f(x) \coloneqq sx + c$ for all $x\in\R$. Then, we have
    \begin{equation*}
        \pd(f_*\mu;\alpha) = |s|\pd(\mu;\alpha).
    \end{equation*}
\end{lemma}
\begin{proof}
    If $s = 0$, then it is clear that both sides of the equality are zero. Assuming that $s \neq 0$, we have
    \begin{align*}
        \pd(f_*\mu;\alpha)
         & = \inf\{\ \diam A \mid A\subset \R \text{ Borel} \text{ and } \mu(f^{-1}(A)) \geq \alpha\ \} \\
         & = \inf\{\ \diam f(B) \mid B\subset \R \text{ Borel} \text{ and } \mu(B) \geq \alpha\ \}      \\
         & = |s|\pd(\mu;\alpha).
    \end{align*}
    This completes the proof.
\end{proof}

\begin{definition}[Prokhorov distance] Let $\mu$ and $\nu$ be two Borel probability measures on a metric space $X$. The \emph{Prokhorov distance} between $\mu$ and $\nu$ is defined as the infimum of $\varepsilon > 0$ such that $\mu(U_\varepsilon(A)) \geq \nu(A) - \varepsilon$ for all Borel subsets $A\subset X$.
\end{definition}

\begin{lemma}[{\cite[Lemma 3.8]{ozawa2015limit}}]
    Let $\mu$ and $\nu$ be two Borel probability measures on a metric space $X$, and let $\alpha\leq 1$ and $\varepsilon > 0$ be real numbers. If $\dpr(\mu,\nu) < \varepsilon$, then\begin{equation*}
        \pd(\mu; \alpha) \leq \pd(\nu; \alpha + \varepsilon) + 2\varepsilon.
    \end{equation*}
\end{lemma}

\begin{definition}[mm-Isomorphism, $\X$]
    Let $X,Y$ be mm-spaces. We say that $X$ and $Y$ are \emph{mm-isomorphic} if there exists an isometry $f\colon X\to Y$ such that $f_*\mu_X = \mu_Y$. Such an isometry is called an \emph{mm-isomorphism}. We denote the set of mm-isomorphism classes of mm-spaces by $\X$.
\end{definition}
\begin{definition}
    Let $X,Y$ be mm-spaces. We say that $X$ dominates $Y$, and we write $Y\prec X$, if there exists a 1-Lipschitz function $f\colon X\to Y$ such that $f_*\mu_X = \mu_Y$.
\end{definition}

The distance known as the \emph{box distance}, denoted by $\Box$, is introduced on $\X$. This distance is something like a metrization of the measured Gromov-Hausdorff convergence. For details, see \cite[§$3\frac{1}{2}$ B]{gromov2007met}.

\begin{definition}[Pyramid]
    A subset $\P\subset \X$ is a \emph{pyramid} if the following conditions (1), (2), and (3) are satisfied.
    \begin{enumerate}
        \item For any $X\in\X$ and $Y\in\P$, if $X\prec Y$, then $X\in\P$.
        \item For any two $X,Y\in\P$, there exists $Z\in\P$ such that $X\prec Z$ and $Y\prec Z$.
        \item $\P$ is non-empty and $\Box$-closed.
    \end{enumerate}
    We denote by $\Pi$ the set of all pyramids.
\end{definition}

\begin{definition}[{1-Measurement and $(1,R)$-Measurement of pyramid}]
    Let $\P$ be a pyramid and $R > 0$ a real number. We define the \emph{$1$-measurement} of $\P$ as
    \begin{align*}
        \M(\P; 1) & \coloneqq \bigl\{f_*\mu_X \mid X\in\P,\ f\in\lip1(X)\bigr\},
    \end{align*}
    and the \emph{$(1,R)$-measurement} of $\P$ as
    \begin{equation*}
        \M(\P; 1, R) \coloneqq \bigl\{f_*\mu_X \mid X\in\P,\ f\in\lip1(X,[-R,R])\bigr\}.
    \end{equation*}
\end{definition}

The set of pyramids $\Pi$ has a natural topology called the weak topology; for which $\Pi$ forms a compactification of $\X$. We do not describe its definition (see \cite[Definition~2.23]{ozawa2015limit}) because we do not use it. The following lemma is enough for the proof of \Cref{thm:lip1minRpartdiam}.
\begin{lemma}[{\cite[Lemma~3.6]{ozawa2015limit}}]\label{lemma:weaktoprokhorov}
    Let $\P$ and $\P_n$, $n=1,2,\ldots$, be pyramids and let $R > 0$. If $\P_n$ converges weakly to $\P$ as $n\to\infty$, then $\M(\P_n;1,R)$ converges to $\M(\P;1,R)$ as $n\to\infty$ in the Hausdorff distance with respect to the Prokhorov distance.
\end{lemma}

\begin{definition}[Observable diameter of pyramid]
    Let $\P$ be a pyramid and $\kappa\in[0,1]$ a real number. We define the ($\kappa$-)\emph{observable diameter} of $\P$ as
    \begin{equation*}
        \od(\P;-\kappa) \coloneqq \sup\{\ \od(X;-\kappa)\mid X\in\P \}.
    \end{equation*}
\end{definition}

\begin{remark}
    The observable diameter defined above is expressed using measurements as
    \begin{align*}
        \od(\P;-\kappa) & = \sup\{\ \pd(f_*\mu_X;1-\kappa)\mid X\in\P \text{ and } f\in\lip1(X)\ \} \\
                        & = \sup\{\ \pd(\mu;1-\kappa)\mid \mu\in\M(\P;1)\ \}.
    \end{align*}
\end{remark}

\begin{lemma}\label{lemma:odrightcontinuous}
    For a pyramid $\P$, the $\kappa$-observable diameter of $\P$ is lower semi-continuous and right-continuous in $\kappa > 0$.
\end{lemma}
\begin{proof}
    From Lemma~\ref{lem:pdLowerSemiContinuous}, $\pd(\mu;1-\kappa)$ is lower semi-continuous in $\kappa > 0$ for any $\mu \in \M(\P;1)$. By the definition of $\od(\P;-\kappa)$, it is also lower semi-continuous. Since $\od(\P;-\kappa)$ is monotonically non-increasing in $\kappa > 0$, this completes the proof.
\end{proof}

\begin{remark}
    In \cite[Definition~3.2]{ozawa2015limit}, the $\kappa$-observable diameter of a pyramid is defined as the right-hand limit of the $(\kappa+\varepsilon)$-observable diameter introduced here, as $\varepsilon\to 0+$. From Lemma~\ref{lemma:odrightcontinuous}, these two definitions coincide to each other.
\end{remark}

\section{Counterexample}
\begin{example}\label{example:odcounterexample}
    Take any real number $R>0$ and any natural number $N\geq 2$. Define an mm-space
    \begin{equation*}
        X_N \coloneqq \{R, 2R, \ldots, 2NR\} \subset \R,
    \end{equation*}
    where $d_{X_N}$ is the Euclidean metric on $\R$ and $\mu_{X_N}$ is the normalized counting measure on $X_N$. Take any $\alpha\in(1/2N,1/N]$ and set $\kappa \coloneqq 1-\alpha$. From Lemma~\ref{lemma:lip1reductpartdiam}, it follows that
    \begin{align*}
        \od(X_N;-\kappa)
         & = \pd({\id}_*\mu_{X_N};\alpha)                                                \\
         & = \inf\{\,\diam A \mid A\subset X_N \text{ and } \mu_{X_N}(A) \geq \alpha\,\} \\
         & = \inf\{\,\diam A \mid A\subset X_N \text{ and } \#A = 2\,\} = R.
    \end{align*}
    We put
    \begin{equation*}
        I \coloneqq \bigl[-(N-1)R,\,(N-1)R\,\bigr] \text{ and } c \coloneqq \frac{\ 2(N-1)\ }{2N-1} < 1.
    \end{equation*}
    Setting $s(x) \coloneqq c(x-R) - (N-1)R$, we see that
    \begin{align*}
        s(X_N) & \subset [s(R),s(2NR)]    \\
               & = [-(N-1)R, (N-1)R] = I.
    \end{align*}
    Lemma~\ref{lemma:preservedistancepartdiam} implies
    \begin{equation*}
        \od(X_N,I;-\kappa)\geq \pd(s_*\mu_{X_N},\alpha) = cR.
    \end{equation*}
    For any $f\in \lip1(X_N,I)$, we have
    \begin{equation*}
        \sum_{n=1}^{N-1} |f(nR) - f((n+1)R)| \leq \diam I = cR(N-1).
    \end{equation*}
    Thus, there exists $n\in\{1,2,\ldots,N-1\}$ such that $|f(nR) - f((n+1)R)| \leq cR$. Since $\mu(\{nR,(n+1)R\}) = 1/N$, we have
    \begin{equation*}
        f_*\mu_{X_N}(\{f(nR),f((n+1)R)\}) = 1/N \geq \alpha
    \end{equation*}
    and
    \begin{equation*}
        \diam\{f(nR),f((n+1)R)\} \leq cR.
    \end{equation*}
    We see that $\pd(f_*\mu_{X_N};\alpha) \leq cR$. From the arbitrariness of $f$, it follows that
    \begin{equation*}
        \od(X_N, I; -\kappa) = cR < R.
    \end{equation*}
\end{example}

\section{Proof of Main Theorem}
\begin{lemma}\label{lemma:lip1preservepartdiam}
    Let $\alpha\in(0,1)$ be a real number, and $\mu$ a Borel probability measure on $\R$ with $\pd(\mu;\alpha) = 1$. Then, there exists a function $f\in\lip1(\R,[-1/\alpha,1/\alpha])$ such that $\pd(f_*\mu;\alpha) = 1$.
\end{lemma}
\begin{proof}
    To construct such a function $f$, we set
    \begin{equation*}
        x_\infty \coloneqq \inf\left\{x\in\R \mmid \mu((x,+\infty)) < \alpha\right\}\in\R
    \end{equation*}
    and define a sequence $(x_n)_{n=0}^\infty\subset[-\infty,+\infty)$ by
    \begin{equation*}
        x_0 \coloneqq -\infty,\quad x_{n+1} \coloneqq \min\bigl\{\,x_\infty,\, \sup\left\{x\in\R \mmid \mu((x_n,x)) < \alpha\right\}\,\bigr\}\in\R
    \end{equation*}
    for $n=0,1,2,\ldots$~. Since $\pd(\mu;\alpha)=1$, we have
    \begin{equation*}
        \mu\bigl((x,x+r)\bigr) < \alpha \text{ for any } x\in\R \text{ and } r\in(0,1),
    \end{equation*}
    and so
    \begin{equation*}
        \min\!\left\{\,x_\infty,\, x_n + 1\,\right\} \leq x_{n+1} \leq x_\infty < +\infty \text{ for } n\geq 1.
    \end{equation*}
    Therefore, the sequence $(x_n)_{n=0}^\infty$ increases strictly until it reaches $x_\infty$ in finite steps, after which it remains constant. Let $N$ be the smallest integer such that $x_N = x_\infty$. Since
    \begin{align*}
        1 & \geq \sum_{n=1}^{N-1} \mu\bigl((x_{n-1}, x_n]\bigr) + \mu\bigl([x_N, +\infty)\bigr) \geq N\alpha,
    \end{align*}
    we have $N \leq 1/\alpha$.

    Next, we construct $f\in\lip1(\R)$. Set the intervals as
    \begin{equation*}
        I_n \coloneqq (x_n-1, x_n+1) \text{ for } n=1,2,\ldots,N,
    \end{equation*}
    and let $A$ be the union of these intervals. We define $f$ as
    \begin{equation*}
        f(x) \coloneqq -N + \int_{-\infty}^x \chi_A \,d\lambda \text { for } x\in\R,
    \end{equation*}
    where $\chi_A$ is the indicator function of $A$, and $\lambda$ the Lebesgue measure.

    Consider the properties of $f$ in order to evaluate the $\alpha$-partial diameter of $f_*\mu$. The function translates elements within each interval $I_n$, ensuring that \begin{equation}\label{eq:translate}
        f(x+a) = f(x)+a \text{ for all } x,a\in\R \text{ with } \{x,x+a\}\subset I_n.
    \end{equation}
    Since $f$ is strictly increasing on the open set $A$ and also non-decreasing over $\R$, it holds that, for any $x\in A$,
    \begin{equation}\label{eq:inverse}
        f^{-1}((-\infty, f(x))) = (-\infty, x), \text{ and } f^{-1}((f(x), +\infty)) = (x, +\infty).
    \end{equation}

    Applying Lemma~\ref{lemma:lip1reductpartdiam}, we observe that $\pd(f_*\mu; \alpha) \leq \pd(\mu; \alpha) = 1$. Let us prove $\pd(f_*\mu; \alpha) \geq 1$. Take any Borel subset $B \subset \R$ with $\diam B < 1$. There exist real numbers $a$ and $b$ such that $B\subset [a,b]$ and $0 < b-a < 1$. It suffices to show that $f_*\mu([a,b]) < \alpha$. There are the following four cases (1)--(4) based on the values of $a$ and $b$ to consider, covering all possibilities:\begin{enumerate}
        \item $b < f(x_1)$.
        \item There exists $n\in\{1,2,\ldots,N\}$ such that $a \leq f(x_n) \leq b$.
        \item There exists $n\in\{1,2,\ldots,N-1\}$ such that $f(x_n) < a < b < f(x_{n+1})$.
        \item $f(x_N) < a$.
    \end{enumerate}

    \begin{description}
        \item[Case (1)] $b < f(x_1)$.
            In this case, there exists $\varepsilon \in (0,1)$ such that $b < f(x_1) - \varepsilon$. Using \eqref{eq:translate} and noting that $\{x_1,\, x_1-\varepsilon\}\subset I_1 = (x_1-1,x_1+1)$, we have $f(x_1-\varepsilon) = f(x_1) - \varepsilon$. From (\ref{eq:inverse}) and the definition of $x_1$, we obtain that\begin{align*}
                f_*\mu([a,b]) & \leq \mu\Big(f^{-1}\big((-\infty, f(x_1) -\varepsilon)\big)\Big) \\
                              & = \mu\big((-\infty, x_1-\varepsilon)\big) < \alpha.
            \end{align*}
        \item[Case (2)] There exists $n\in\{1,2,\ldots,N\}$ such that $a\leq f(x_n)\leq b$. In this case, since $b-a<1$, we have
            \begin{equation*}
                a' \coloneqq a - f(x_n) \in (-1,0\!\;], \text{ and } b' \coloneqq b - f(x_n) \in [\!\;0,1).
            \end{equation*}
            Since $\{a'+x_n$, $x_n$, $b'+x_n\}\in I_n = (x_n-1,x_n+1)$ and by (\ref{eq:translate}), it follows that\begin{equation*}
                f(x_n+a') = f(x_n) + a', \text{ and } f(x_n+b') = f(x_n) + b'.
            \end{equation*}
            By (\ref{eq:inverse}), we obtain\begin{equation*}
                f^{-1}([a,b]) = f^{-1}\Big(\big[f(x_n+a'), f(x_n+b')\big]\Big) = [x_n+a', x_n+b'],
            \end{equation*}
            and so
            \begin{equation*}
                \diam\Big(f^{-1}\big([a,b]\big)\Big) = b'-a' = b-a < 1.
            \end{equation*}
            Since $\pd(\mu;\alpha)=1$, we have
            \begin{equation*}
                f_*\mu\big([a,b]\big) = \mu\Big(f^{-1}\big([a,b]\big)\Big) < \alpha.
            \end{equation*}
        \item[Case (3)] There exists $n\in\{1,2,\ldots,N-1\}$ such that $f(x_n) < a < b < f(x_{n+1})$. In this case, there exists $\varepsilon\in(0,1)$ such that $b < f(x_{n+1}) - \varepsilon$. Since $x_n\in A$ and $\{x_{n+1}-\varepsilon,\,x_{n+1}\}\subset I_{n+1}$, we obtain
            \begin{equation*}
                f^{-1}([a,b]) \subset f^{-1}\big((f(x_n),\ f(x_{n+1}) - \varepsilon)\big) = (x_n,\ x_{n+1} - \varepsilon).
            \end{equation*}
            From the definition of $x_{n+1}$, this implies
            \begin{equation*}
                f_*\mu([a,b]) \leq \mu((x_n,\ x_{n+1} - \varepsilon)) < \alpha.
            \end{equation*}
        \item[Case (4)] $f(x_N) < a$.
            In this case, there exists $\varepsilon \in (0,1)$ such that $f(x_N) + \varepsilon < a$. Using the definition of $x_N$ and noting $\{x_N,x_N+\varepsilon\}\subset I_N$, we have
            \begin{align*}
                f_*\mu([a,b]) & \leq \mu\big(f^{-1}((f(x_N) + \varepsilon, +\infty))\big) \\
                              & = \mu((x_N + \varepsilon, +\infty)) < \alpha.
            \end{align*}
    \end{description}

    Under each scenario, we show that $f_*\mu(B) \leq f_*\mu([a,b]) < \alpha$, which implies $\pd(f_*\mu; \alpha) \geq 1$. Since $\inf f(\R) \geq -N \geq -1/\alpha$ and
    \begin{equation*}
        \sup f(\R) \leq -N + \sum_{n=1}^N \lambda(I_n) = N \leq 1/\alpha,
    \end{equation*}
    the values of $f$ are confined to the interval $[-1/\alpha,1/\alpha]$. This completes the proof.
\end{proof}

\begin{proof}[Proof of Theorem~\ref{thm:lip1minRpartdiam}]
    Define $r \coloneqq \pd(\mu; \alpha)$. If $r=0$, then the function $f \equiv 0$ renders the equality trivially true. Assume $r > 0$ and define a function $s$ as $s(x) \coloneqq x/r$ for $x\in\R$. Observe that $\pd(s_*\mu; \alpha) = 1$. By Lemma~\ref{lemma:lip1preservepartdiam}, there exists a function $g\in \lip1(\R,[-1/\alpha, 1/\alpha])$ such that $\pd(g_*s_*\mu; \alpha) = 1$. We define a function $e$ as $e(x) \coloneqq \min\{R, r\}\cdot x$ for $x\in\R$. Since $\min\{R, r\}/r \leq 1$, we have $f \coloneqq e \circ g\circ s \in \lip1(\R)$. Lemma~\ref{lemma:preservedistancepartdiam} shows that
    \begin{align*}
        \pd(f_*\mu; \alpha)
         & = \min\{R, r\}\cdot \pd(g_*s_*\mu; \alpha) \\
         & = \min\{R, \pd(\mu; \alpha)\}.
    \end{align*}
    Assessing the range of $f$, we see that
    \begin{equation*}
        f(\R) = e(g\circ s(\R)) \subset [-R/\alpha, R/\alpha].
    \end{equation*}
    This completes the proof.
\end{proof}

\begin{lemma}\label{lemma:od-pq-measurement}
    Let $\P,\Q$ be pyramids and $\kappa\in(0,1), \varepsilon\in(0,1-\kappa)$ real numbers. If we have
    \begin{equation}\label{eq:pq-measurement}
        \M\left(\P;1,\frac{R}{1-(\kappa+\varepsilon)}\right) \subset U_\varepsilon\left(\M\left(\Q; 1\right)\right),
    \end{equation}
    then \begin{equation*}
        \min\{R,\od(\P; -(\kappa + \varepsilon))\} \leq \od(\Q; -\kappa) + 2\varepsilon.
    \end{equation*}
\end{lemma}
\begin{proof}
    Take any $X\in\P$ and $f\in\lip1(X)$. By Lemma~\ref{lemma:lip1reductpartdiam}, there exists
    \begin{equation*}
        g\in\lip1\left(\R,\left[-\frac{R}{1-(\kappa+\varepsilon)},\frac{R}{1-(\kappa+\varepsilon)}\right]\right)
    \end{equation*}
    such that \begin{equation*}
        \min\left\{R,\ \pd(f_*\mu_X; 1-(\kappa+\varepsilon))\right\} = \pd(g_*f_*\mu_X; 1-(\kappa+\varepsilon)).
    \end{equation*}
    From \eqref{eq:pq-measurement}, there exists $\nu\in\M(\Q;1)$ such that $\dpr(g_*f_*\mu_X, \nu) \leq \varepsilon$, so that
    \begin{align*}
        \min\{R,\pd(f_*\mu_X; 1-(\kappa+\varepsilon))\}
         & = \pd(g_*f_*\mu_X; 1-(\kappa+\varepsilon)) \\
         & \leq \pd(\nu; 1-\kappa) + 2\varepsilon     \\
         & \leq \od(\Q; -\kappa) + 2\varepsilon.
    \end{align*}
    The arbitrariness of $f$ and $X$ yields
    \begin{equation*}
        \min\{R,\od(\P; -(\kappa + \varepsilon))\} \leq \od(\Q; -\kappa) + 2\varepsilon.
    \end{equation*}
    This completes the proof.
\end{proof}
\begin{proof}[Proof of Theorem~\ref{thm:pyramid_od_limit}]
    The theorem follows from the same discussion as in \cite{ozawa2015limit} using Theorem 1.2. For the completeness we provide the details.

    If $\kappa \geq 1$, all terms in the equalities above are obviously zero. Assume $\kappa < 1$. Take any two real numbers $R>0$ and $\varepsilon > 0$ such that $\kappa + 2\varepsilon < 1$. Lemma~\ref{lemma:weaktoprokhorov} implies that there exists a natural number $N$ such that
    \begin{equation*}
        \hausp{\dpr}\!\left(\M\!\left(\P;1,\frac{R}{1-(\kappa+2\varepsilon)}\right),\ \M\!\left(\P_n; 1,\frac{R}{1-(\kappa+2\varepsilon)}\right)\right) < \varepsilon
    \end{equation*}
    for all $n\geq N$, where $\hausp{\dpr}$ is the Hausdorff distance with respect to the Prokhorov distance. By Lemma~\ref{lemma:od-pq-measurement} and since
    \begin{align*}
        \M\!\left(\P;1,\frac{R}{1-(\kappa+2\varepsilon)}\right)
         & \subset U_\varepsilon\!\left(\M\!\left(\P_n; 1,\frac{R}{1-(\kappa+2\varepsilon)}\right)\right) \\
         & \subset U_\varepsilon(\M(\P_n; 1)) \text{ for all } n \geq N,
    \end{align*}
    we have
    \begin{equation*}
        \min\{R,\od(\P; -(\kappa + 2\varepsilon))\} \leq \liminf_{n\to\infty} \od(\P_n; -(\kappa+\varepsilon)) + 2\varepsilon.
    \end{equation*}
    Similarly, from Lemma~\ref{lemma:od-pq-measurement} and since
    \begin{align*}
        \M\!\left(\P_n; 1,\frac{R}{1-(\kappa+\varepsilon)}\right)
         & \subset \M\!\left(\P_n; 1,\frac{R}{1-(\kappa+2\varepsilon)}\right)     \\
         & \subset U_\varepsilon(\M\left(\P; 1\right)) \text{ for all } n \geq N,
    \end{align*}
    we have also
    \begin{equation*}
        \limsup_{n\to\infty} \min\{R,\od(\P_n; -(\kappa + \varepsilon))\} \leq \od(\P; -\kappa) + 2\varepsilon.
    \end{equation*}
    The arbitrariness of $R$ yields
    \begin{align*}
        \od(\P; -(\kappa + 2\varepsilon))
         & \leq \liminf_{n\to\infty} \od(\P_n; -(\kappa+\varepsilon)) + 2\varepsilon \\
         & \leq \limsup_{n\to\infty} \od(\P_n; -(\kappa+\varepsilon)) + 2\varepsilon \\
         & \leq \od(\P; -\kappa) + 4\varepsilon.
    \end{align*}
    By Lemma~\ref{lemma:odrightcontinuous}, the limit of those inequalities as $\varepsilon\to 0$ proves the equalities. This completes the proof of the theorem.
\end{proof}

\begin{proof}[Proof of \Cref{prop:lip1minRpartdiamIsSharp}]
    Take any real number $R>0$ and $n=2,3,\ldots$. We set $\kappa \coloneqq 1-1/n$. By \Cref{example:odcounterexample}, there exists an mm-space $X_n$ such that
    \begin{equation*}
        \od(X_n; -\kappa_n) = R > \od(X_n, [-(n-1)R, (n-1)R]; -\kappa_n),
    \end{equation*}
    which implies
    \begin{equation*}
        Q\left(R,\left[-\!\left(\frac{1}{1-\kappa_n}-1\right)\!R,\ \left(\frac{1}{1-\kappa_n}-1\right)\!R\right],-\kappa_n\right) > 1.
    \end{equation*}
    Thus, we have
    \begin{align*}
         & \limsup_{\kappa\to 1-}\ \sup \left\{ \frac{2R}{1-\kappa} - \diam I \mid I\subset \R \text{ Borel}, Q(R,I,\kappa) \leq 1\right\} \\
         & \leq \limsup_{\kappa\to 1-} \left(\frac{2R}{1-\kappa} - 2R\left(\frac{1}{1-\kappa}-1\right)\right) \leq 2R.
    \end{align*}
    This completes the proof.
\end{proof}

\bibliographystyle{abbrv}
\bibliography{obsdiam}
\end{document}